\documentclass[11pt, reqno]{amsart}
\usepackage{lipsum}
\usepackage{amsthm,amsmath,amssymb}
\usepackage{mathrsfs}
\usepackage[colorlinks, linkcolor=black, citecolor=blue, linktocpage,urlcolor=blue]{hyperref}
\usepackage[colorlinks, linkcolor=black, citecolor=blue, linktocpage,urlcolor=blue]{hyperref}
\numberwithin{equation}{section}
\theoremstyle{definition}
\newtheorem{thm}{Theorem}[section]
\newtheorem{definition}[thm]{Definition}
\newtheorem{proposition}[thm]{Proposition}

\newtheorem{lemma}[thm]{Lemma}
\newtheorem{corollary}[thm]{Corollary}

\newtheorem{remark}[thm]{Remark}

\begin{document}
	\setlength{\abovedisplayskip}{10pt}
	\setlength{\belowdisplayskip}{5pt}
	
	\newcommand{\Ext}{\bigwedge\nolimits}
	\newcommand{\Div}{\operatorname{div}}
	\newcommand{\Hol} {\operatorname{Hol}}
	\newcommand{\diam} {\operatorname{diam}}
	\newcommand{\Scal} {\operatorname{Scal}}
	\newcommand{\scal} {\operatorname{scal}}
	\newcommand{\Ric} {\operatorname{Ric}}
	\newcommand{\Hess} {\operatorname{Hess}}
	\newcommand{\grad} {\operatorname{grad}}
	\newcommand{\Rm} {\operatorname{Rm}}
	\newcommand{ \Rmzero } {\mathring{\Rm}}
	\newcommand{\Rc} {\operatorname{Rc}}
	\newcommand{\Curv} {S_{B}^{2}\left( \mathfrak{so}(n) \right) }
	\newcommand{ \tr } {\operatorname{tr}}
	\newcommand{ \Riczero } {\mathring{\Ric}}
	\newcommand{ \Ad } {\operatorname{Ad}}
	\newcommand{ \dist } {\operatorname{dist}}
	\newcommand{ \rank } {\operatorname{rank}}
	\newcommand{\Vol}{\operatorname{Vol}}
	\newcommand{\dVol}{\operatorname{dVol}}
	\newcommand{ \zitieren }[1]{ \hspace{-3mm} \cite{#1}}
	\newcommand{ \pr }{\operatorname{pr}}
	\newcommand{\diag}{\operatorname{diag}}
	\newcommand{\Lagr}{\mathcal{L}}
	\newcommand{\av}{\operatorname{av}}
	\newcommand{ \floor }[1]{ \lfloor #1 \rfloor }
	\newcommand{ \ceil }[1]{ \lceil #1 \rceil }
	\newcommand{\Sym} {\operatorname{Sym}}
	\newcommand{\bcirc}{ \ \bar{\circ} \ }
	\newcommand{\sign}[1]{\operatorname{sign}(#1)}
	\newcommand{\cone}{\operatorname{cone}}
	\newcommand{\pbd}{\varphi_{bar}^{\delta}}
	\newcommand{\End}{\operatorname{End}}
	
	\renewcommand{\labelenumi}{(\alph{enumi})}
	\newtheorem{maintheorem}{Theorem}[]
	\renewcommand*{\themaintheorem}{\Alph{maintheorem}}
	\newtheorem*{remark*}{Remark}
	\title{Some vanishing theorems on $(p,q)$ double form and curvature operator of the second kind}
	
	
	\author{Hai-Ping Fu}
	\address{Department of Mathematics, Nanchang University, Nanchang 330031, People’s Republic of China}
	\curraddr{}
	\email{mathfu@126.com}
	\thanks{Supported in part by National Natural Science Foundations of China \#12461008 and 12271069,  Jiangxi Province
		Natural Science Foundation of China \#20202ACB201001, Jiangxi Province Graduate Student Innovation Special Fund Project \#YC2025-B037.}
	
	\author{Yao Lu}
	\address{Department of Mathematics, Nanchang University, Nanchang 330031, People’s Republic of China}
	\curraddr{}
	\email{luyao@email.ncu.edu.cn}
	\thanks{}
	
	
	
	\date{}
	
	\dedicatory{}
	
	\begin{abstract}
		We establish a Bochner formula for double forms in terms of the curvature operator of the second kind. As an application, we prove that a complete Riemannian manifold of dimension $n \ge 4$
		with harmonic Weyl tensor and $\frac{3(n-1)}{4}$-nonnegative curvature operator of the second kind is either globally conformally equivalent to a space of positive
		constant curvature or is isometric to a flat manifold. We prove vanishing theorems for the Lichn\'{e}rowicz Laplacian $\Delta$ on $(p,q)$ double forms. These generalize recent results of Nienhaus-Petersen-Wink \cite{NPW23} and Dai-Fu-Lu-Yang \cite{DF24,DFY24,FL1,FLD}.
	\end{abstract}
	\maketitle
	
	\section{Introduction}
	The Lichn\'{e}rowicz Laplacians acts on $(0,p)$ tensor is defined as
	\[{\Delta }_{L}T={\nabla }^*\nabla T+cRic_L(T)\]
	for a suitable constant $c>0.$ For example, the Hodge Laplacian on differential forms
	corresponds to $c=1$, and symmetric $(0,2)$-tensors as well as the curvature tensor is $c = \frac{1}{2}$. A tensor $T$ is called harmonic if $\Delta_T = 0$.
	
	In \cite{PW21}, Petersen and Wink give a new method to control the curvature term of Lichnerowicz Laplacians on $(0, p)$-tensors $T$, and obtain vanishing and
	estimation theorem for the $p$th Betti number of closed $n$-dimensional Riemannian manifolds with a lower bound on the average of the lowest $(n-p)$-eigenvalues of the curvature operator. By investigating the Bochner formulas for various geometric quantities, there are a large number of geometric rigidity theorems about eigenvalues of the curvature operator. see \cite{Tra17,BG24,CMR24,PW22,HZ25}
	
	The curvature term of Lichnerowicz Laplacians on $(0, k)$-tensors $T$, can also be control by curvature operator of the second kind. In \cite{NPW23}, Nienhaus-Petersen-Wink establish that the curvature term of the Lichnerowicz Laplacian admits an expression in terms of the curvature operator of the second kind and obtain a vanishing theorem for Betti numbers under an eigenvalue pinching condition on the curvature operator of the second kind. Based on this approach, Dai and Fu \cite{DF24} studied the Bochner formula for the curvature tensor in Einstein manifolds and obtained a sphere theorem concerning the eigenvalue pinching of the curvature operator of the second kind. Recently, this result has been appropriately generalized on manifold with harmonic curvature and harmonic Weyl curvature. See \cite{DFY24,FLD,FL1}. More recently results can found in \cite{CGT23,Li22,Li23a,Li24,Li24b,Li25}
	
	A $(p, q)$-double form $\omega$ on a Riemannian manifold $(M, g)$ is a covariant tensor of rank $p+q$ which is alternating with respect to the first $p$ and the last $q$ arguments. The double form $\omega$ can beregarded as a vector valued differential $p$-form over $M$ or alternatively as a vector valued differential $q$-form over $M$. In \cite{Labbi3}, Labbi introduced a new product on double forms to establish index-free formulas for the curvature terms in the Weitzenb\"{o}ck formulas corresponding to the Laplacians $\Delta_L$ and proved vanishing theorems for $\Delta_L$ on symmetric double forms of arbitrary order.
	
	In this paper, we study the curvature term of the Lichn\'{e}rowicz Laplacians acting on $(p,q)$ double form and derive a new Bochner
	formula in terms of the curvature operator of the second kind. We obtain vanishing theorems for doubles forms of arbitrary order.
	
	First, We first establish a vanishing theorem for symmetric $(p, p)$-double forms. Set
	
	\begin{equation}\label{constant}
		C(p,n)=\left\{ \begin{aligned}
			& \frac{n(n-1)(n+2)}{2(n^2-2)}\quad for \quad p=1, \\ 
			& \frac{3(n-1)}{4}\quad for \quad p=2, \\ 
			& \frac{n-p+1}{2} \quad for \quad p\ge 3.
		\end{aligned} \right.
	\end{equation}
	\begin{maintheorem}\label{theorem-A}
		Let $(M, g)$ be a closed and connected Riemannian manifold of dimension $n \ge 3$ and
		let $\omega$ be a harmonic $(p, p)$ symmetric double form on $M$, satisfying the first
		Bianchi identity and its first contraction vanish. If the curvature operator of the second kind $\mathring{R}$ is $C(p,n)$-positive (resp. $C(p,n)$-nonnegative), then the double form $\omega $ vanishing (resp. $\omega $ must be parallel).
	\end{maintheorem}
	As an application, Theorem A improves several previous rigidity results. For example, taking the Weyl tensor on Riemannian manifolds with
	harmonic curvature, we obtain the following improvement of the result in \cite{DF24,FLD}.
	\begin{corollary}
		Let $(M, g)$ be an $n(\ge 4)$-dimensional complete Riemannian manifold with
		harmonic curvature. If the curvature operator of the second kind $\mathring{R}$ is $\frac{3(n-1)}{4}$-nonnegative, then $M$ must be of constant curvature.
	\end{corollary} 
	More generally, we can improve the result in \cite{FL1} as follows:
	\begin{corollary}
		Let $(M, g)$ be an $n(\ge 4)$-dimensional complete Riemannian manifold with harmonic Weyl tensor. If the curvature operator of the second kind $\mathring{R}$ is $\frac{3(n-1)}{4}$-nonnegative and  the Weyl tensor of $(M, g)$ has bounded norm, then $M$ is either globally conformally equivalent to a space of positive constant curvature or is isometric to a flat manifold.
	\end{corollary} 
	
	Second, we obtain the following vanishing theorem for general $(p,q)$ double forms. Set $$C(p,q,n)=\frac{\frac{3}{2}[n(p+q)-p^2-q^2]+3q}{A(p,q,n)},$$
	where \begin{align*}
		&A(p,q,n)\\
		&=\left\{ \begin{aligned}
			& p+q+2min\{p,q\}+\frac{(p+q)(n-3(p+q))}{n+2}+2q(p+1)\frac{n^2+n-4}{n^2+n-2}\quad p,q\le 2 \\ 
			& p+3q-\frac{(p+q)^2}{n+2}+\frac{(n-2(p+q))(2n+2+nq)}{n(n+2)}+2q(p+1)\frac{n^2+n-4}{n^2+n-2}\quad p\ge 3,q\le 2 \\ 
			&3p+q-\frac{(p+q)^2}{n+2}+\frac{(n-2(p+q))(2n+2+np)}{n(n+2)}+2q(p+1)\frac{n^2+n-4}{n^2+n-2}\quad  p\le 2,q\ge 3 \\ 
			& p+q+2\min\{p,q\}
			-\frac{(p+q)^2}{n+2}
			+\frac{4(n+1)(n-2p-2q)}{n(n+2)}
			+2q(p+1)\frac{n^2+n-4}{n^2+n-2}\quad p,q\ge 3.
		\end{aligned} \right.
	\end{align*}
	\begin{maintheorem}\label{pq}
		Let $(M, g)$ be a closed and connected Riemannian manifold of dimension $n \ge 3$ and
		let $\omega$ be a harmonic $(p, q)$ double form on $M$ with $1 \le p+q \le \frac{n}{2}$, satisfying the first
		Bianchi identity and its first contraction vanish. If the curvature operator of the second $\mathring{R}$ is $C(p,q,n)$-positive (resp. $C(p,q,n)$-nonnegative), then the double form $\omega $ vanishing (resp. $\omega $ must be parallel).
	\end{maintheorem}
	From Theorem~\ref{pq}, we recover Theorem A and C in \cite{NPW23} where we identify a $p$-form $\theta$ with $\omega= \theta \otimes 1$ a $(p, 0)$ double form.		
	
	It follows from \cite{Labbi} that any symmetric $(p,p)$ double form has an orthogonal decomposition $\omega=\sum_{i=0}^{p}g^i\omega_{p-i}$, where $g_i$ are the exterior powers of the metric $(1, 1)$ double form $g$ and the $\omega_k$ are $(k, k)$ double forms such that $\boldsymbol{c}\omega_k = 0$. Based on  Theorem~\ref{theorem-A}, we obtain the following rigidity result. Set
	$${C}'(j,n)=\left\{ \begin{aligned}
		& \frac{n(n-1)(n+2)}{2(n^2-2)}\quad for \quad j=1,2 \\ 
		& C(j,n)\quad for \quad j\ge 3 \\ 
	\end{aligned} \right.$$

	\begin{maintheorem}
		Let $(M, g)$ be a closed and connected Riemannian manifold of dimension $n \ge 3$ and
		let $\omega$ be a harmonic $(p, p)$ symmetric double form on $M$ with $1 \le p \le \frac{n}{2}$, satisfying the first
		Bianchi identity. Let $j$ be an integer such that $1 \le j \le p$. If the curvature operator of the second kind $\mathring{R}$ is
		$C'(j,n)$-positive then $\boldsymbol{c}^{p-j}\omega = \lambda g^j$ for some constant $\lambda$. In particular,\\
		a) If $\mathring{R}$ is $C'(p,n)$-positive then $\omega$ has constant sectional curvature.\\
		b) If $\mathring{R}$ is $\frac{n(n-1)(n+2)}{2(n^2-2)}$-positive then $\boldsymbol{c}^{p-1}\omega = \lambda g$ for some constant $\lambda$.
	\end{maintheorem}
	It follows from \cite{Labbi3} that the Gauss–Kronecker tensor $R^k$ of $2k$-Thorpe manifolds and $2k$-Anti-Thorpe manifolds is harmonic, and hence the above theorem  applies to them.
	
	\begin{maintheorem}
		Let $(M, g)$ be a closed and connected Riemannian manifold of dimension $n \ge 3$. Let
		$\omega $ be a harmonic $(p, p)$ symmetric double form on $M$ with $2 \le p \le \frac{n}{2}$, satisfying the first and second Bianchi identities such that its first contraction satisfies  $\boldsymbol{c}\omega =\lambda g^{p-1}$ for some function $\lambda$ on $M$. Suppose the curvature operator of the second kind $\mathring{R}$ is $C(p,n)$-positive (resp. $C(p,n)$-nonnegative), then the double form $\omega$ has constant sectional curvature (resp. $\omega $ must be parallel).
	\end{maintheorem}
	As described in \cite{Labbi2}, a Riemannian manifold is called hyper $2k$-Einstein if its Riemann curvature tensor, viewed as a $(2,2)$ double form, satisfies $\boldsymbol{c} R^k = \lambda g^{2k-1}$. We have a corollary analogous to that in \cite{Labbi3}.
	\begin{corollary}
		Let $(M, g)$ be a closed and connected hyper $2k$-Einstein Riemannian $n$-manifold with
		$2 \le 2k \le \frac{n}{2}$. If $\mathring{R}$ is $p$-positive (resp. $p$-nonnegative) for $1 \le p \le C(2k,n)$, then $M$ has constant $(2k)$-th Thorpe sectional curvature, that is $R^k = constant.g^{2k}$ (resp. $R^k$ must be parallel).
	\end{corollary}
	
	\section{Preliminaries}
	In this section, we recall some basic notions and identities concerning
	double forms, the first Bianchi sum, and the curvature operator of the
	second kind. For further details on the curvature operator of the second kind, we refer to
	\cite{Li22,Li23a,NPW23}; see also
	\cite{Kulkarni,Labbi,Labbi1} for a comprehensive treatment of double forms.
	
	\subsection{Double Forms and the First Bianchi Sum}
	
	Let $(V,g)$ be an $n$-dimensional Euclidean vector space. Denote by
	$\Lambda V^*$ the exterior algebra of the dual space $V^*$, namely,
	\[
	\Lambda V^*
	=
	\bigoplus_{p=0}^n \Lambda^p V^*.
	\]
	The space of \emph{double forms} on $V$ is defined by
	\[
	DV^*
	:=
	\Lambda V^*\otimes \Lambda V^*
	=
	\bigoplus_{p,q=0}^n D^{p,q}V^*,
	\qquad
	D^{p,q}V^*
	:=
	\Lambda^pV^*\otimes\Lambda^qV^*.
	\]
	Thus, a $(p,q)$ double form is a covariant tensor which is alternating
	in its first $p$ arguments and separately alternating in its last
	$q$ arguments.
	
	The space $DV^*$ carries a natural bi-graded associative algebra
	structure, obtained by taking the exterior product in each factor.
	More precisely, if
	\[
	\omega_1=\theta_1\otimes\theta_2\in D^{p,q}V^*,
	\qquad
	\omega_2=\theta_3\otimes\theta_4\in D^{r,s}V^*,
	\]
	then their product is defined by
	\[
	\omega_1\cdot\omega_2
	:=
	(\theta_1\wedge\theta_3)
	\otimes
	(\theta_2\wedge\theta_4)
	\in D^{p+r,q+s}V^*.
	\]
	We refer to this product as the \emph{exterior product of double forms}.
	The dot ``$\cdot$'' will be used to distinguish this product from the
	ordinary tensor product.
	
	Fix an orthonormal basis $\{e_i\}_{i=1}^n$ of $V$, and let
	$\{e^i\}_{i=1}^n$ be the corresponding dual basis of $V^*$. 
	With respect to this convention, the collection $\left\{ e^{i_1}\wedge\cdots\wedge e^{i_k}: 1\leq i_1<\cdots<i_k\leq n\right\}$
	forms an orthogonal basis of $\Lambda^kV^*$. Consequently,
	\[
	\left\{
	(e^{i_1}\wedge\cdots\wedge e^{i_p})
	\otimes
	(e^{j_1}\wedge\cdots\wedge e^{j_q})
	:
	\begin{array}{l}
		i_1<\cdots<i_p,\\
		j_1<\cdots<j_q
	\end{array}
	\right\}
	\]
	is an orthogonal basis of $D^{p,q}V^*$.	
	Accordingly, every $\omega\in D^{p,q}V^*$ can be written uniquely as
	\[
	\omega
	=
	\sum_{\substack{
			i_1<\cdots<i_p\\
			j_1<\cdots<j_q
	}}
	\omega_{i_1\ldots i_p,j_1\ldots j_q}\,
	e^{i_1}\wedge\cdots\wedge e^{i_p}
	\otimes
	e^{j_1}\wedge\cdots\wedge e^{j_q}.
	\]
	
	We regard
	$\Lambda^p V^*$ as the subspace of $\otimes^p V^*$ consisting of
	alternating tensors and equip $\otimes^{p+q}V^*$ with the standard
	tensor inner product induced by the metric $g$. This induces an inner
	product, and hence a norm, on $D^{p,q}V^*$. Therefore
	\[
	|\omega|^2=\sum_{i_1,\ldots,i_p,j_1,\ldots,j_q}\omega_{i_1\ldots i_p,j_1\ldots j_q}^{\,2}=p!\,q!\sum_{\substack{
			i_1<\cdots<i_p\\
			j_1<\cdots<j_q
	}}
	\omega_{i_1\ldots i_p,j_1\ldots j_q}^{\,2}.
	\]
	Throughout the paper, we use this tensor norm convention for double
	forms.
	
	For $\alpha\in V^*$, the \emph{interior product}
	$\iota_\alpha:\Lambda^kV^*\longrightarrow\Lambda^{k-1}V^*,$ defined by
	$$\iota_\alpha(\omega)(X_1,X_2,...,X_{p-1})=\omega(\alpha, X_1,X_2,...,X_{p-1}).$$	
	The same construction extends naturally to double forms. Namely, for $\alpha=e^i\otimes e^j,$
	if $\omega=\omega_1\otimes\omega_2$
	are decomposable double forms, then 
	\[
	\iota_{\alpha}
	(\omega)
	=
	(\iota_{e^i}\omega_1)
	\otimes
	(\iota_{e^j}\omega_2).
	\]
	The first contraction map $\boldsymbol{c}$ of double forms is defined as
	\[
	\boldsymbol{c}\omega=\iota_g(\omega)
	\]
	for metric $g$ seen as a $(1, 1)$-double form. 
	
	In terms of components, if
	$\omega\in D^{p,q}V^*$, and $p=0$ or $q=0$, set $c\omega=0$. If both $p,q> 1$ then 
	\[(\boldsymbol{c}\omega)_{i_1\ldots i_{p-1},j_1\ldots j_{q-1}}=\sum_{i=1}^{n}{\omega_{ii_1\ldots i_{p-1},ij_1\ldots j_{q-1}}}\]
	
	We next recall the \emph{first Bianchi sum}. Let
	$\omega=\omega_1\otimes\omega_2\in D^{p,q}V^*$ be a decomposable
	double form. Its first Bianchi sum is defined by
	\[
	\mathcal{B}\omega
	:=
	\sum_{i=1}^n
	e^i\wedge\omega_1
	\otimes
	\iota_{e_i}\omega_2
	\in D^{p+1,q-1}V^*.
	\]
	The operator $\mathcal{B}$ is then extended to arbitrary double forms by linearity.
	Thus,
	\[
	\mathcal{B}:D^{p,q}V^*\longrightarrow D^{p+1,q-1}V^*.
	\]
	We say that a double form $\omega$ satisfies the \emph{first Bianchi identity} if $\mathcal{B}\omega=0.$
	
	In terms of components, if
	$\omega\in D^{p,q}V^*$ and $p=0$ or $q=0$, set $\mathcal{B}\omega=0$. If both $p,q> 1$ then  then
	\begin{align}\label{Bianchi1}
		(\mathcal{B}\omega)_{i_1\ldots i_{p+1},j_1\ldots j_{q-1}}
		&=
		\sum_{\alpha=1}^{p+1}
		(-1)^{\alpha-1}
		\omega_{
			i_1\ldots\widehat{i_\alpha}\ldots i_{p+1},
			\,i_\alpha j_1\ldots j_{q-1}
		}.
	\end{align}
	Under the tensor norm, we have
	\begin{align}\label{Bianchi}
		\left| \mathcal{B}\omega  \right|^2=(p+1)\left| \omega  \right|^2-p(p+1)\sum\limits_{i,j,i_2\ldots i_p,j_2\ldots j_q }{{\omega }_{ii_2\ldots i_p,jj_2\ldots j_q}{\omega }_{ji_2\ldots i_p,ij_2\ldots j_q}}.
	\end{align}

	\subsection{Curvature operator of the second kind}
	
	In a Riemannian manifold $(M^n,g)$, the Riemannian curvature tensor $R$  induces two symmetric linear operators $\hat{R}$ and $\mathring{R}$. Denote by $V$  the tangent space of this Riemannian manifold, by ${\Lambda }^2V$  the space of skew symmetric $2$-tensor over $V$, and by $S_0^2(V)$  the space of trace-free symmetric $2$-tensor over $V$. The curvature operator of the first kind is defined by
	\begin{align*}
		& \hat{R}:{\Lambda }^2V\to {\Lambda }^2V \\ 
		& \hat{R}(e_i\wedge e_j)=\frac{1}{2}\sum\limits_{k,l}{R_{ijkl}e_k\wedge e_l},
	\end{align*}
	and the curvature operator of the second kind is defined in \cite{CGT23, NPW23} by
	$$\mathring{R}={\Pr}_{s_0^2(V)}\circ \left.\bar{R} \right|_{S_0^2(V)},$$
	where the operator $\bar{R}$ act on $S^2(V)$ induced from Riemannian curvature tensor $R$ is defined by
	$$\bar{R}(e_i\odot e_j)=\sum\limits_{k,l}{R_{kijl}e_k\odot e_l},$$
	where $e_i\odot e_j=e_i\otimes e_j+e_j\otimes e_i$. 
	The inner product on $S^2(V)$ is defined by
	\[\left\langle A, B \right\rangle= tr(A^T B),\]
	Then $\frac{1}{\sqrt{2}}\{e_i\odot e_j\}_{1\le i<j\le n},\frac{1}{2}\{e_i\odot e_i\}_{i=1,\dots,n}$ is an orthonormal basis for $S^2(V)$. The operator $\mathring{R}$ can be regarded as the curvature operator $\bar{R}$ and its image restricts on $S_0^2(V)$. 
	
	By identifying symmetric $(0,2)$-tensors with self-adjoint endomorphisms of $V$, Nienhaus, Petersen and Wink give the following definition in \cite{NPW23}.
	\begin{definition}
		Let $T^{(0,k)}(V)$ denote the space of $(0,k)$-tensor on $V$. For $S\in S^2(V)$ and $T\in T^{(0,k)}(V),$ we define
		\begin{align*}
			& S:T^{(0,k)}(V)\to T^{(0,k)}(V) \\ 
			& \left(ST\right)(X_1,\cdots ,X_k)=\sum\limits_{i=1}^{k}{T(X_1,\cdots ,SX_i,\cdots ,X_k)} 
		\end{align*}
		and define $T^{S^2}\in T^{(0,k)}(V)\otimes S^2(V)$ by
		$$\left\langle T^{S^2}(X_1,\cdots ,X_k),S \right\rangle =\left(ST\right)(X_1,\cdots ,X_k).$$
	\end{definition} 
	Hence, if $\left\{ S^\alpha \right\}$ is an orthonormal basis for $S_0^2(V)$, then
	\[T^{S_0^2}=\sum\limits_{\alpha=1}^{N}{S^\alpha T\otimes S^\alpha},\]
	where $N=dim(S_0^2(V))=\frac{(n-1)(n+2)}{2}$.
	
	Let $\{{\lambda }_\alpha\}$ denote the eigenvalues of the operator $\mathring{R}\,$  and $\left\{ S^\alpha \right\}$ be the corresponding orthonormal eigenbasis of $S_{0}^2(V)$. In what follows, the eigenvalues $\{{\lambda}_\alpha\}$ are assumed to be arranged in non‑decreasing order. Define the action of $\mathring{R}\,$ on the space $T^{S_{0}^{2}}$ as
	\[\mathring{R}\,(T^{S_{0}^{2}})=\sum\limits_{\alpha}{S^\alpha T\otimes \mathring{R}\,(S^{\alpha}),}\]
	then
	\[\left\langle \mathring{R}\,(T^{S_{0}^{2}}),T^{S_{0}^{2}} \right\rangle =\sum\limits_{\alpha}{{\lambda }_\alpha\left| S^{\alpha}T \right|^2}.\]
	\subsection{Weight principle}
	To analyze the individual terms of Bochner formula, we apply the weighted-sum calculus developed in \cite{NPW23}, which allows us to estimate finite weighted sums with nonnegative weights.
	\begin{definition}(\cite[Definition 3.1]{NPW23})
		Let $\{\omega_{\alpha}\}_{\alpha=1}^{N}$ be the nonnegative weights of any finite weighted sums. Define
		\[
		\Omega=\max_{1\leq \alpha\leq N}\omega_{\alpha} \quad \mathrm{and} \quad  S=\sum_{\alpha=1}^{N}\omega_{\alpha}\,.
		\]
		We call $S$ the total weight and $\Omega$ the highest weight.
	\end{definition}
	Let $\{\lambda_{\alpha}\}_{\alpha=1}^{N}$ be the eigenvalues of $\mathring{R}\,$. Following the notation in \cite{NPW23}, we write $[\mathring{R},\Omega,S]$ to denote any finite weighted sums $\sum_{\alpha=1}^{N}\omega_{\alpha}\lambda_{\alpha}$ in terms of $\{\lambda_{\alpha}\}_{\alpha=1}^{N}$ whose weights satisfy highest weight $\Omega$ and total weight $S$.
	In particular, if $F(R)$ is a (geometric) quantity depending on curvature tensor $R$ and $\mathring{R}$ is curvature operator of the second kind, then we will write
	$$F(R) \ge \left[\mathring{R},\Omega, S\right]$$
	provided $F(R)$ is bounded from below by a weighted sum in terms of the eigenvalues of $\mathring{R}$ with highest weight $\Omega$ and total weight $S$.
	\begin{lemma}\label{Lem 3.4}(\cite[Lemma 3.4]{NPW23})
		Let  $\{\lambda_{\alpha}\}_{\alpha=1}^{N}$ denote the ordered eigenvalues of $\mathring{R}$. Then, for any integer $1\leq m\leq N$, we have
		$$\left[\mathring{R},\Omega,S\right]\geq\left(S-m\Omega\right)\lambda_{m+1}+\Omega\sum_{\alpha=1}^{m}\lambda_{\alpha}.$$
		Therefore, $\left[\mathring{R},\Omega,S\right]\ge 0$ if $\mathring{R}$ is $\frac{S}{\Omega}$-nonnegative.
	\end{lemma}
	
	The following lemma will be useful in our proof.
	\begin{lemma}\label{weight principle}(\cite[Lemma 3.3]{NPW23})
		Let $\left[\mathring{R}, \Omega, S\right], \left[\mathring{R}, \widetilde{\Omega}, \widetilde{S}\right]$ denote weighted sums of eigenvalues of $\mathring{R}$ with highest weights $\Omega, \widetilde{\Omega}$ and total weights $S, \widetilde{S}$, respectively.
		\begin{enumerate}
			\item If $c>0$, then 
			\begin{align*}
				\left[\mathring{R}, c \Omega, c S\right] = c \cdot \left[\mathring{R}, \Omega, S\right].
			\end{align*}
			\item If $\Omega \leq \widetilde{\Omega},$ then 
			\begin{align*}
				\left[ \mathring{R} , \Omega, S\right] \geq \left[ \mathring{R} , \widetilde{\Omega}, S\right].
			\end{align*}
			\item \begin{align*}
				\left[\mathring{R}, \Omega, S\right] + \left[\mathring{R}, \widetilde{\Omega}, \widetilde{S}\right] \geq \left[\mathring{R}, \Omega + \widetilde{\Omega}, S + \widetilde{S}\right].
			\end{align*}
		\end{enumerate}
	\end{lemma}

	\section{A BOCHNER FORMULA FOR $(p,p)$ DOUBLE FORM}
	In this section, we give a Bochner formula for symmetric double form and proof Theorem A. The Lichn\'{e}rowicz Laplacians is defined as
	\[{\Delta }_{L}T={\nabla }^*\nabla T+cRic_L(T)\]
	for a suitable constant $c>0.$
	Applying the Lichnerowicz formula for $(0,p)$-tensors established in \cite[1.143]{Einstein} to $(p,q)$-double forms, the curvature term is given by
	\begin{align}\label{Bochner}
		(\Ric_L \omega)_{i_1 \dots i_p, \, j_1 \dots j_q} &= \sum_{a=1}^p R_{i_a}{}^k \omega_{i_1 \dots \underset{(a)}{k} \dots i_p, \, j_1 \dots j_q} + \sum_{b=1}^q R_{j_b}{}^k \omega_{i_1 \dots i_p, \, j_1 \dots \underset{(b)}{k} \dots j_q} \nonumber\\ 
		&\quad - 2 \sum_{1 \le a < b \le p} R_{i_a}{}^m{}_{i_b}{}^n \omega_{i_1 \dots \underset{(a)}{m} \dots \underset{(b)}{n} \dots i_p, \, j_1 \dots j_q} \nonumber\\ 
		&\quad - 2 \sum_{1 \le a < b \le q} R_{j_a}{}^m{}_{j_b}{}^n \omega_{i_1 \dots i_p, \, j_1 \dots \underset{(a)}{m} \dots \underset{(b)}{n} \dots j_q} \nonumber\\ 
		&\quad - 2 \sum_{a=1}^p \sum_{b=1}^q R_{i_a}{}^m{}_{j_b}{}^n \omega_{i_1 \dots \underset{(a)}{m} \dots i_p, \, j_1 \dots \underset{(b)}{n} \dots j_q}. 
	\end{align}
	
	We shall often use the shorthand notation
	\begin{align*}
		I=(i_1,\ldots,i_p),I'=(i_1,\ldots,i_{p-1});\qquad respectively,\quad I'' \ and\  I''',\\
		J=(j_1,\ldots,j_q),J'=(j_1,\ldots,j_{q-1});\qquad respectively,\quad J'' \ and \ J'''
	\end{align*}
	for the corresponding multi-indices.
	\begin{proposition}
		Let $\omega$ be a symmetric $(p,p)$ double form satisfying the first Bianchi identity, then
		\begin{equation}\label{eq:bochner-Sijkl}
			g(\Ric_{L}(\omega),\omega)
			=
			\frac12
			\sum_{I,J}
			\left\langle
			\mathring R(S^{I,J}),
			S^{I,J}
			\right\rangle,
		\end{equation}
		where for $I=(i_1,\ldots,i_p),J=(j_1,\ldots,j_p)$
		
		\begin{equation}\label{eq:Sijkl-def}
			S^{I,J}=\sum_{s}
			\sum\limits_{\alpha =1}^{p}\omega_{i_1\ldots \underset{(\alpha)}{s}\ldots i_p,J}\,e_s\odot e_{i_\alpha}-
			\sum_{s}\sum\limits_{\alpha =1}^{p}\omega_{I,j_1\ldots \underset{(\alpha)}{s}\ldots j_p}\,e_s\odot e_{j_\alpha}
		\end{equation}
	\end{proposition}
	\begin{proof}
		Let $\omega$ be a symmetric $(p,p)$ double form, from (\ref{Bochner}) we have
		\begin{align}\label{Bochner1} 
			\frac{1}{2}g(\Ric_L(\omega ),\omega)=pR_{st}{\omega }_{sI',J}{\omega }_{tI',J}-p(p-1){R_{ijkl}}{\omega }_{jlI'',J}{\omega }_{ikI'',J}-p^2R_{ijkl}{\omega }_{iI',kJ'}{\omega }_{jI',lJ'}.
		\end{align}
		
		It follows from (\ref{eq:Sijkl-def}) that $S^{I,J}$ is a trace-free symmetric $(0,2)$-tensor. Therefore,
		 $$\left\langle
		 \mathring R(S^{I,J}),
		 S^{I,J}
		 \right\rangle=\left\langle
		 \bar R(S^{I,J}),
		 S^{I,J}
		 \right\rangle.$$
		
		A direct computation using the symmetries of the curvature tensor,
		together with	the first Bianchi identity yields
		
		\begin{align}\label{SIJ}
			& \frac{1}{2p}\sum\limits_{I,J}{\left\langle \mathring R(S^{I,J}),S^{I,J} \right\rangle } \nonumber\\ 
			=&{\omega }_{si_2\ldots i_p,J}{\omega }_{ti_2\ldots i_p,J}\left\langle \bar R(e_s\odot e_{i_1}),e_t\odot e_{i_1} \right\rangle\nonumber\\
			& +\sum\limits_{\alpha =2}^{p}{{\omega}_{si_2\ldots i_p,J}{\omega }_{i_1\ldots \underset{(\alpha)}{t}\ldots i_p,J}}\left\langle \bar R(e_s\odot e_{i_1}),e_t\odot e_{i_{\alpha}} \right\rangle-\sum\limits_{\alpha =1}^{p}{{\omega }_{si_2\ldots i_p,J}{\omega }_{I,j_1\ldots \underset{(\alpha)}{t}\ldots j_p}\left\langle \bar R(e_s\odot e_{i_1}),e_t\odot e_{j_{\alpha }} \right\rangle } \nonumber\\ 
			=&2R_{st}{\omega}_{sI',J}{\omega }_{tI',J}+2(p-1)(R_{stji}+R_{sjti}){\omega }_{sjI'',J}{\omega }_{itI'',J}-2p(R_{stji}+R_{sjti}){\omega }_{sI',jJ'}{\omega }_{iI',tJ'} \nonumber\\ 
			=&2R_{st}{\omega }_{sI',J}{\omega }_{tI',J}+2(p-1)R_{stji}{\omega }_{sjI'',J}{\omega }_{itI'',J}+2(p-1)R_{sjti}{\omega}_{sjI'',J}{\omega }_{itI'',J}\nonumber\\
			&-2p(R_{sijt}+R_{sjti}){\omega }_{sI',jJ'}{\omega }_{iI',tJ'}-2pR_{sjti}{\omega }_{sI',jJ'}{\omega }_{iI',tJ'} \nonumber\\ 
			=&2R_{st}{\omega }_{sI',J}{\omega }_{tI',J}-2(p-1)R_{stji}{\omega }_{sjI'',J}{\omega }_{tiI'',J}-2pR_{sijt}{\omega }_{sI',jJ'}{\omega }_{iI',tJ'}\nonumber\\
			&-2(p-1)R_{sjit}{\omega }_{sjI'',J}{\omega }_{itI'',J}+4pR_{sjit}{\omega }_{sI',jJ'}{\omega }_{iI',tJ'}
		\end{align}
		where the first and second equalities follow from 
		$${\omega}_{i_1\ldots \underset{(\alpha)}{s}\ldots i_p,J}{\omega }_{i_1\ldots \underset{(\beta)}{t}\ldots i_p,J}={{\omega }_{si_{\beta }{I}/{(i_{\alpha},i_{\beta})}\;}}{{\omega}_{i_{\alpha}t{I}/{(i_{\alpha },i_{\beta })}\;}}:={{\omega }_{sjI'',J}}{{\omega }_{itI'',J}},$$
		and
		$${\omega}_{i_1\ldots \underset{(\alpha)}{s}\ldots i_p,J}{\omega }_{I,j_1\ldots \underset{(\beta)}{t}\ldots j_p}={{\omega }_{s{I}/{(i_{\alpha})},j_{\beta}{J}/{(j_{\beta })}\;}}{{\omega}_{i_{\alpha}{I}/{(i_{\alpha})},t{J}/{(j_{\beta })}\;}}:={{\omega }_{sI',jJ'}}{{\omega }_{iI',tJ'}}.$$
		
		A comparison of (\ref{SIJ})  with (\ref{Bochner1})  yields
		\begin{equation*}
			\frac{1}{2}g(\Ric_{L}(\omega ),\omega )=\frac{1}{4}\sum\limits_{I,J}{\left\langle \mathring R(S^{I,J}),S^{I,J} \right\rangle }+p(p-1)R_{sjit}{\omega}_{sjI'',J}{\omega }_{itI'',J}-2p^2R_{sjit}{\omega }_{sI',jJ'}{\omega }_{iI',tJ'}.
		\end{equation*}
		Proposition 3.1 follows immediately from Proposition~\ref{estimate1}.
	\end{proof}
	\begin{proposition}\label{estimate1}
		Let $\omega$ be a symmetric $(p,p)$ double form satisfying the first Bianchi identity, then
		\[(p-1)R_{sjit}{\omega}_{sj{I}'',J}{\omega }_{it{I}'',J}=2pR_{sjit}{\omega}_{sI',jJ'}{\omega}_{iI',tJ'},\]
	\end{proposition}	
	\begin{proof}
		Applying (\ref{Bianchi1}) to ${\omega }_{itI'',J}$ while also using the symmetry of $\omega $ yields 
		\begin{align*}
			0={\left( \mathcal{B}\omega  \right)}_{i_{p-2}J,it{I}'''}&={\omega}_{J,i_{p-2}itI'''}+\sum\limits_{\alpha }{(-1)^{\alpha }{\omega }_{i_{p-2}{J}/{j_{\alpha}},j_{\alpha }itI'''\;}}\\
			&=
			{\omega}_{i_{p-2}itI''',J}+\sum\limits_{\alpha }{(-1)^{\alpha }{\omega }_{j_{\alpha }itI''',i_{p-2}{J}/{j_{\alpha}}\;}}\\
			&=(-1)^{p-3}{\omega}_{itI'',J}+\sum\limits_{\alpha }{(-1)^{\alpha }{\omega }_{itj_{\alpha }I''',i_{p-2}{J}/{j_{\alpha}}\;}}. 
		\end{align*}
		Multiplying both sides by  $R_{sjit}{\omega }_{sjI'',J}$ gives 
		\begin{eqnarray}\label{Bianchi2}
			R_{sjit}{\omega }_{sjI'',J}{\omega }_{itI'',J}&=&\sum\limits_{\alpha }{(-1)^{\alpha-1}R_{sjit}{\omega }_{sji_{p-2}I''',J}{\omega }_{itj_{\alpha}I''',i_{p-2}{J}/{j_{\alpha}}\;}}\nonumber\\
			&=&\sum\limits_{\alpha }{R_{sjit}{\omega }_{sji_{p-2}I''',j_{\alpha}{J}/{j_{\alpha}}}{\omega }_{itj_{\alpha}I''',i_{p-2}{J}/{j_{\alpha}}\;}}\nonumber\\
			&=&pR_{sjit}{\omega }_{sjlI''',mJ'}{\omega}_{itmI''',lJ'}.
		\end{eqnarray}
		
		On the other hand, for the term $R_{sjit}{\omega }_{sI',jJ'}{\omega }_{iI',tJ'}$, by using the symmetries of the curvature tensor we have
		\[R_{sjit}{\omega }_{sI',jJ'}{\omega }_{iI',tJ'}=\frac{1}{4}R_{sjit}({\omega }_{sI',jJ'}-{\omega }_{jI',sJ'})({\omega }_{iI',tJ'}-{\omega }_{tI',iJ'}).\]
		From (\ref{Bianchi1})	
		\[0=\left( \mathcal{B}\omega  \right)_{jsI',J'}={\omega }_{sI',jJ'}-{\omega }_{jI',sJ'}+\sum\limits_{\alpha =1}^{p-1}{\left(-1 \right)^{\alpha-1}{\omega}_{js{I'}/{(i_{\alpha })}\;,i_{\alpha}J'}},\]
		that is 	
		\[{\omega }_{sI',jJ'}-{\omega }_{jI',sJ'}=\sum\limits_{\alpha =1}^{p-1}{\left( -1 \right)^{\alpha }{\omega}_{js{I'}/{(i_{\alpha })}\;,i_{\alpha }J'}}.\]
		Similarly,	
		\[{\omega }_{iI',tJ'}-{\omega }_{tI',iJ'}=\sum\limits_{\alpha =1}^{p-1}{\left( -1 \right)^{\alpha}{\omega }_{ti{I'}/{(i_{\alpha })}\;,i_{\alpha }J'}}.\]
		Therefore,
		\begin{align*}
			4R_{sjit}{\omega }_{sI',jJ'}{\omega }_{iI',tJ'}&=\sum\limits_{\alpha =1}^{p-1}{R_{sjit}{\omega }_{js{I'}/{(i_{\alpha })}\;,i_{\alpha }J'}{\omega }_{ti{I'}/{(i_{\alpha })}\;,i_{\alpha}J'}}+\sum\limits_{\alpha \ne \beta }{(-1)^{\alpha +\beta}R_{sjit}{\omega }_{js{I'}/{(i_{\alpha})}\;,i_{\alpha }J'}{\omega }_{ti{I'}/{(i_{\beta })}\;,i_{\beta }J'}} \\ 
			& =(p-1){{R}_{sjit}}{{\omega }_{sj{I}'',J}}{{\omega }_{it{I}'',J}}-(p-1)(p-2){{R}_{sjit}}{{\omega }_{sjl{I}''',m{J}'}}{{\omega }_{itm{I}''',l{J}'}}.\\ 
		\end{align*}
		Substituting (\ref{Bianchi2}) into the above expression completes the proof.
	\end{proof}
	
	Let $\{S_\alpha\}_{\alpha=1}^N$ be an orthonormal eigenbasis of $\mathring R$ with corresponding eigenvalues $\lambda_\alpha$.
	Write
	\[
	S^{I,J}
	=
	\sum_{\alpha=1}^N
	(S^{I,J})_\alpha S_\alpha,
	\qquad
	(S^{I,J})_\alpha
	=
	\langle S^{I,J},S_\alpha\rangle,
	\]
	then from \eqref{eq:bochner-Sijkl},
	\begin{equation}\label{eq:bochner-weight}
		\begin{split}
			g(\Ric_L(\omega),\omega)
			&=
			\frac12
			\sum_{\alpha=1}^N
			\lambda_\alpha
			\sum_{I,J}
			\bigl((S^{I,J})_\alpha\bigr)^2.
		\end{split}
	\end{equation}
	To apply the weight principle of Lemma~\ref{Lem 3.4}, we need to find the total weight of $\sum_{\alpha,I,J}
	\bigl((S^{I,J})_\alpha\bigr)^2$ and estimate the highest weight of $\sum_{I,J}
	\bigl((S^{I,J})_\alpha\bigr)^2$.
	Since $\{S_\alpha\}$ is an
	orthonormal basis of $S^2_0(V)$,
	\[
	\sum_{\alpha=1}^N
	\bigl((S^{I,J})_\alpha\bigr)^2
	=
	|S^{I,J}|^2.
	\]
	Moreover, recall that $\mathring{R}=id$ if $M$ is round sphere, hence setting $R_{ijkl}=\delta_{ik}\delta_{jl}-\delta_{il}\delta_{jk}$ in (\ref{SIJ}) and combining with Proposition~\ref{estimate1} we get
	\begin{equation}\label{eq:Sijkl-total-norm}
		\sum_{I,J}|S^{I,J}|^2
		=
		4p(n-p+1)|\omega|^2-4p^2|\boldsymbol{c}\omega|^2,
	\end{equation}
	where we use (\ref{Bianchi}) to get $\omega_{sI',jJ'}\omega_{jI',sJ'}=\frac1p|\omega|^2.$
	
	It remains to estimate the highest weight of $\sum_{I,J}((S^{I,J})_\alpha)^2$. Fix $\alpha$, at the point
	under consideration, choose an orthonormal basis such that
	\[
	S_\alpha
	=
	\operatorname{diag}(a_1,\ldots,a_n).
	\]
	Since $S_\alpha$ is trace-free and $|S_\alpha|=1$,
	\begin{equation*}\label{eq:a-normalization}
		\sum_{i=1}^n a_i=0,
		\qquad
		\sum_{i=1}^n a_i^2=1.
	\end{equation*}
	From \eqref{eq:Sijkl-def}, we have
	\begin{equation*}
		\begin{split}
			(S^{I,J})_\alpha
			&=\langle S^{I,J},S_\alpha\rangle \\
			&=2(\sum\limits_{i\in I}{a_i}-\sum\limits_{j\in J}{a_j}){\omega}_{I,J}.
		\end{split}
	\end{equation*}
	Thus
	\begin{equation}\label{omega-diagonal}
		\sum_{I,J}\bigl((S^{I,J})_\alpha\bigr)^2	=4\sum\limits_{I,J}{(\sum\limits_{i\in I}{a_i}-\sum\limits_{j\in J}{a_j})^2\omega _{I,J}^2}.
	\end{equation}
	Suppose $|I\cap J|=r$, then by Cauchy-Schwarz inequality
	$$(\sum\limits_{i\in I}{a_i}-\sum\limits_{j\in J}{a_j})^2\le 2(p-r)\sum\limits_{i\in {I\cup J}/{I\cap J}\;}{a_{i}^2}\le 2(p-r) \le2p.$$
	Therefore,
	\begin{equation}\label{eq:omega-diagonal2}
		\sum_{I,J}\bigl((S^{I,J})_\alpha\bigr)^2\le 8p|\omega|^2.    
	\end{equation}

	In the case $p = 2$, we can refine the argument by applying the Bianchi identity. Let $I=(i,j),J=(k,l)$, from \eqref{eq:Sijkl-def} we have
	\begin{equation*}\label{eq:Sijkl-component}
		\begin{split}
			(S^{I,J})_\alpha
			&=\langle S^{I,J},S_\alpha\rangle \\
			&=2(a_i+a_j-a_k-a_l)\omega_{ij,kl}.
		\end{split}
	\end{equation*}
	Thus
	\begin{equation}\label{eq:omega-diagonal}
		\sum_{I,J}\bigl((S^{I,J})_\alpha\bigr)^2	=4\sum_{i,j,k,l}(a_i+a_j-a_k-a_l)^2\omega_{ij,kl}^2.
	\end{equation}
	
	\begin{proposition}
		If $\omega$ satisfies the first Bianchi identity
		\begin{equation}\label{eq:key-weight}
			\sum_{i,j,k,l}
			(a_i+a_j-a_k-a_l)^2\omega_{ij,kl}^2
			\leq
			\frac83|\omega|^2.
		\end{equation}
	\end{proposition}
	\begin{proof}
		For fixed $i,j,k,l$, if $|\{i,j\}\cap \{k,l\}|=1$, then 
		\begin{equation}\label{eq1}
			(a_i+a_j-a_k-a_l)^2\omega_{ij,kl}^2\le 2\omega_{ij,kl}^2\le\frac83\omega_{ij,kl}^2.
		\end{equation}
		
		If $|\{i,j\}\cap \{k,l\}|=0$, that is, $i,j,k,l$ are pairwise distinct. By the first Bianchi identity,				
		\begin{equation}\label{eq:xyz-bianchi}
			\omega_{ij,kl}+\omega_{ik,lj}+\omega_{il,jk}=0.
		\end{equation}
		Denote
		\[
		\begin{split}
			p&=a_i+a_j-a_k-a_l,\\
			q&=a_i+a_k-a_l-a_j,\\
			r&=a_i+a_l-a_j-a_k.
		\end{split}
		\]
		A direct computation gives
		\begin{equation}\label{eq:pqr}
			\begin{split}
				p^2+q^2+r^2&=4(a_i^2+a_j^2+a_k^2+a_l^2)-(a_i+a_j+a_k+a_l)^2\\
				&\leq4(a_i^2+a_j^2+a_k^2+a_l^2)\leq 4.
			\end{split}
		\end{equation}	
		On the other hand, \eqref{eq:xyz-bianchi} implies
		\[
		\omega_{ij,kl}^2\leq 2(\omega_{ik,lj}^2+\omega_{il,jk}^2),
		\]
		and hence
		\[
		\omega_{ij,kl}^2\leq\frac23(\omega_{ij,kl}^2+\omega_{ik,lj}^2+\omega_{il,jk}^2).
		\]
		The same argument gives
		\[
		\omega_{ik,lj}^2,\omega_{il,jk}^2\leq \frac23(\omega_{ij,kl}^2+\omega_{ik,lj}^2+\omega_{il,jk}^2).
		\]
		Combining this with \eqref{eq:pqr}, we obtain
		\begin{align*}
			p^2\omega_{ij,kl}^2+q^2\omega_{ik,lj}^2+r^2\omega_{il,jk}^2
			&\leq
			\frac23(\omega_{ij,kl}^2+\omega_{ik,lj}^2+\omega_{il,jk}^2)
			(p^2+q^2+r^2)\\
			&\leq
			\frac83(\omega_{ij,kl}^2+\omega_{ik,lj}^2+\omega_{il,jk}^2).
		\end{align*}
		Summing over $\{i,j,k,l\}$ yields
		\[ (a_i+a_j-a_k-a_l)^2\omega_{ij,kl}^2\le \frac83\omega_{ij,kl}^2,\]
		this together with (\ref{eq1}) proves \eqref{eq:key-weight}.
	\end{proof}
	\begin{remark}
		The constant $\frac83$ in \eqref{eq:key-weight} is optimal. Indeed, take $n=4$, $$a_1=a_2=\frac12,a_3=a_4=-\frac12,$$  and 
		$$\omega_{12,34}=2,\omega_{13,24}=1,\omega_{14,23}=-1,\omega_{ijkl}=0\quad\text{for}\quad\{i,j,k,l\}\ne \{1,2,3,4\}.$$
		By direct computation, one can verify the equality of Proposition 3.3.
	\end{remark}
	
	\begin{proof}[Proof of Theorem A]
		For $p=1$, we can directly compute from (\ref{Bochner})
		\begin{align*}
			g(\Ric_L(\omega ),\omega )&=2R_{ij}{\omega }_{i,k}{\omega }_{j,k}-2{R_{ijkl}}{\omega }_{i,k}{\omega}_{j,l} \\ 
			& =2{R_{ij}}{{\omega }_{i,k}}{\omega}_{j,k}+2\left\langle \mathring{R}(\omega ),\omega  \right\rangle,
		\end{align*}
		elementary Lagrange multiplier argument we have
		\[\left\langle \mathring{R}(\omega ),\omega  \right\rangle\ge[\mathring{R},\frac{n^2+n-4}{n^2+n-2},1]|\omega|^2. \]
		According to Lemma~\ref{weight principle} and (\ref{Ricci1}), it follows that  
		\begin{align*}
			g(\Ric_L(\omega ),\omega)\ge&2[\mathring{R},\frac{n}{n+2},n-1]|\omega|^2+2[\mathring{R},\frac{n^2+n-4}{n^2+n-2},1]|\omega|^2\\
			\ge&2[\mathring{R},\frac{2(n^2-2)}{n^2+n-2},n]|\omega|^2.
		\end{align*}
		
		For $p=2$. If a harmonic symmetric $(2, 2)$ double form $\omega$ on $M$ satisfies the first
		Bianchi identity and its first contraction vanish, it is easy to see from \eqref{eq:Sijkl-total-norm} that 
		\begin{equation}\label{eq:Sijkl-total-norm1}
			\sum_{I,J}|S^{I,J}|^2
			=
			8(n-1)|\omega|^2.
		\end{equation}
		From \eqref{eq:omega-diagonal} and \eqref{eq:key-weight}, we get
		\begin{equation}\label{eq:omega-diagonal1}
			\sum_{I,J}\bigl((S^{I,J})_\alpha\bigr)^2	\leq\frac{32}{3}|\omega|^2.
		\end{equation}
		By the weighted principle and Proposition 3.1,  from \eqref{eq:Sijkl-total-norm1} and \eqref{eq:omega-diagonal1} we have 
		\[
		g(\Ric_L({\omega}),\omega)
		\geq
		\left[
		\mathring R,
		\frac{16}{3},
		4(n-1)
		\right]|\omega|^2.
		\]
		
		For $p\ge 3$, it follows from (\ref{eq:Sijkl-total-norm}) and (\ref{eq:omega-diagonal2}),
		\[
		g(\Ric_L({\omega}),\omega)
		\geq
		\left[
		\mathring R,
		4p,
		2p(n-p+1)
		\right]|\omega|^2.
		\]
		Hence $g(\Ric_L(\omega ),\omega )\ge 0$ if $\mathring{R}$ is $C(p,n)$-nonnegative.						
		The conclusion follows from maximum principle and the Bochner technique.
	\end{proof}
	
	We present an application of Theorem A to  manifolds with harmonic  curvature or harmonic Weyl tensor, thereby extending the result of \cite{DF24, DFY24, FL1, FLD}.
	\begin{corollary}(Corollary 1.1). Let $(M, g)$ be an $n(\ge 4)$-dimensional complete Riemannian
		manifold with harmonic curvature. If the curvature operator of the second kind $\mathring{R}$ is $\frac{3(n-1)}{4}$-nonnegative, then $M$ must be of constant curvature.
	\end{corollary}
	\begin{corollary}(Corollary 1.2). Let $(M, g)$ be an $n(\ge 4)$-dimensional complete Riemannian
		manifold with harmonic Weyl tensor. If the curvature operator of the second kind $\mathring{R}$ is$\frac{3(n-1)}{4}$-nonnegative and the Weyl tensor of $(M, g)$ has bounded norm, then M is either globally
		conformally equivalent to a space of positive constant curvature or is isometric to a flat
		manifold.
	\end{corollary}
	\begin{proof}[Proof of Corollary 1.1]
		Riemannian curvature tensor  is divergence free if and only if its Ricci tensor is a Codazzi tensor, and in this case its scalar curvature is constant. Hence the traceless Ricci tensor is harmonic because its Ricci tensor is divergence free. Denote by $E$ its traceless Ricci tensor and by $s$ its scalar curvature.
		
		If $s\equiv0$, then $M$ is flat since $\mathring{R}$ is $\frac{3(n-1)}{4}$-nonnegative.	Otherwise, $s=constant>0$,   i.e.,   $M$ is non-flat. By part $(2)$ of Proposition 4.1 in \cite{Li24}, we have that $Ric(M)\geq\frac{s}{n(n+1)}>0$ because  $\mathring{R}$ with $\frac{3(n-1)}{4}$-nonnegative is $n$-positive. So $M$ is compact by Myers Theorem.
		
		For manifolds with harmonic weyl tensor,  we have
		\[\frac{1}{2}\Delta \left| W \right|^2=\left| \nabla W \right|^2+\frac12g (\Ric_{L} (W),W) \ge 0. \]
		By the maximum principle, from the above we have $\nabla W=0$. According to Derdziński
		and Roter’s result \cite{DR} that Riemannian manifolds with  $\nabla W=0$ are either
		locally conformally flat or locally symmetric, we see that $(M, g)$ is either locally
		conformally flat or locally symmetric.
		Case I, $M$ is locally conformally flat with positive constant scalar curvature. Let ${a_i}$ be the eigenvalue of Ricci tensor, then we obtain that
		\begin{align}\label{laplacian1}
			\left\langle\Delta E,E \right\rangle &=E_{ij}E_{ij,tt}=E_{ij}E_{it,jt}\nonumber\\
			&=E_{ij}(E_{it,tj}+R_{jtsi}E_{st}+R_{jtst}E_{is})\nonumber\\
			&=\frac{n}{n-2}R_{ij}E_{it}E_{jt}-\frac{s}{(n-2)(n-1)}E_{ij}E_{ij}\nonumber\\
			&=\frac{n}{n-2}R_{ij}R_{jt}R_{ti}-\frac{(2n-1)s}{(n-1)(n-2)}R_{ij}R_{ji}+\frac{s^3}{(n-1)(n-2)}\nonumber\\
			&=\frac{n}{n-2}\sum a_i^3-\frac{(2n-1)s}{(n-1)(n-2)}\sum a_i^2+\frac{s^3}{(n-1)(n-2)}\nonumber\\
			&=\frac{n}{n-2}\left(\frac{(2n-1)s}{n(n-1)}\right)^3\left[\sum\left(\frac{n(n-1)a_i}{(2n-1)s}\right)^3-\sum\left(\frac{n(n-1)a_i}{(2n-1)s}\right)^2\right]\nonumber\\
			&+\frac{s^3}{(n-1)(n-2)}.\\\nonumber
		\end{align}		
		It follows from \cite[Lemma 3.3]{FL} and (\ref{laplacian1}) that the minimal points of $\left\langle\Delta E,E \right\rangle$ are $(\frac{s}{n}, \frac{s}{n},\cdots, \frac{s}{n})$ and $(0, \frac{s}{n-1},\cdots, \frac{s}{n-1})$, and that $\left\langle\Delta E,E \right\rangle\geq0$.  Hence  from \begin{equation*}\label{laplacian}
			\frac{1}{2}\Delta \left| E \right|^2=\left| \nabla E \right|^2+\left\langle\Delta E,E \right\rangle,
		\end{equation*} we deduce that the trace-less Ricci tensor $E$ is parallel. 
		Since $(M, g)$ is  locally irreducible by Theorem 1.8 in \cite{Li24}, it follows that $(M, g)$ is isometric to a space of positive constant curvature.
		
		Case II, $M$ is locally symmetric with positive constant scalar curvature. Since $(M, g)$ is  locally irreducible by Theorem 1.8 in \cite{Li24}, $(M, g)$ is an Einstein manifold.		It follows from Theorem B of \cite{NPW23} that $M$ is a rational homology sphere. Wolf \cite[Theorem 1]{Wo} provides a complete classification of compact symmetric spaces that are rational homology spheres. Aside from spheres themselves, the only simply connected example is the space $SU(3)/SO(3)$, which is shown in \cite[Example 4.5]{NPW23} to be $9$-positive. It then follows that the manifold must have positive constant curvature. \\
	\end{proof}
	
	\begin{proof}[Proof of Corollary 1.2]
		The difference between the proofs of Corollaries 1.1 and 1.2 arises from the fact that the scalar curvature in Corollary 1.2 is not constant; the remaining parts of the proof are identical. The only difference occurs in the case where $(M, g)$
		is locally conformally flat.

		From the proof of Corollary 1.1,  $Ric(M)\geq\frac{s}{n(n+1)}\geq0$. Based on Corollary 1.1 in \cite{LW}, we obtain that $(M, g)$ is compact. It follows from the work of Chow \cite{C} that there exists a conformal metric $\widetilde{g}$  of 
		$g$ such that the scalar curvature $s_{\widetilde{g}}$ is constant and the Ricci curvature $Ric_{\widetilde{g}}$  is nonnegative. And $(M, \widetilde{g})$
		is locally conformally flat. From the proof of Case I in Corollary 1.1, $(M, g)$ is globally
		conformally equivalent to a space of positive constant curvature. An alternative proof is given as follows.
		
		If $(M, g)$ is locally conformally flat, then by the classification theorem in \cite{CH,Z} and the fact that $(M, g)$ is locally irreducible, it follows that $(M, g)$ is globally conformally equivalent to a space of positive constant curvature.\\\\
	\end{proof}
	\section{A Bochner formula for $(p,q)$ double form}
	In this section, we give the following Bochner formula for $(p,q)$ double form in terms of the curvature
	operator of the second kind.
	\begin{proposition}\label{proposition1}
		let $\omega$ be a $(p,q)$-double form. With respect to an orthonormal basis, the curvature term in the Bochner formula satisfies
		\[\begin{aligned}
			\frac{3}{2}g(\Ric_{L}(\omega),\omega)=&g(\mathring{R}({\omega}^{S_0^2}),{\omega }^{S_0^2})+\left( p-\frac{2(p+q)p}{n} \right)R_{ij}{\omega }_{iI',J}{\omega }_{jI',J}\\
			&+\left( q-\frac{2(p+q)q}{n} \right)R_{ij}{\omega }_{I,iJ'}{\omega }_{I,jJ'}+\frac{(p+q)^2}{n^2}scal\left| \omega  \right|^2\\
			&-2pqR_{kijl}{{\omega }_{iI',lJ'}}{\omega }_{jI',kJ'}-2pqR_{ijkl}{\omega }_{iI',kJ'}{\omega }_{jI',lJ'}. 
		\end{aligned}\]
		In particular, if $\boldsymbol{c}\omega=0$ we have
		\begin{align}\label{pqbochner}
			\frac{3}{2}g(\Ric_L(\omega ),\omega )=&g(\mathring{R}({\omega}^{S_0^2}),{\omega }^{S_0^2})+\left( p-\frac{2(p+q)p}{n} \right)R_{ij}{\omega }_{iI',J}{\omega }_{jI',J}\nonumber\\
			&+\left( q-\frac{2(p+q)q}{n} \right)R_{ij}{\omega }_{I,iJ'}{\omega }_{I,jJ'}+\frac{(p+q)^2}{n^2}scal\left| \omega  \right|^2\\
			&+pq\sum_{I',J'}{g(\mathring{R}(W^{I',J'}),W^{I',J'})}\nonumber,
		\end{align}
		where
		$$W^{I',J'}=\sum_{i,j}\frac{1}{2}({\omega }_{iI',jJ'}+{\omega }_{jI',iJ'})e_i\odot e_j.$$
	\end{proposition}
	Proposition~\ref{proposition1} generalizes Proposition 2.1 in \cite{NPW23}. It can be regarded as the case of $(0,p+q)$-tensors in the remark 2.5 of \cite{NPW23} and we employ an approach analogous to that in \cite{NPW23} and split into the proofs of several propositions.
	
	\begin{proposition}\label{compute}
		For an orthonormal basis $\{e_1,\ldots, e_n\}$, set $e_i \bar{\circ} e_j = \frac{1}{2}(e_i \otimes e_j + e_j \otimes e_i)$. Then, every $(p,q)$-double form $\omega$ satisfies
		$$g((e_i\bar{\circ }e_j)\omega ,\omega )
		=p\sum\limits_{I',J}{{\omega }_{iI',J}{\omega }_{jI',J}}
		+q\sum\limits_{I,J'}{{\omega }_{I,iJ'}{\omega }_{I,jJ'}}$$
		and
		\begin{eqnarray*}
			&&g((e_i\bar{\circ }e_j)\omega ,(e_k\bar{\circ }e_l)\omega ) \\ 
			&=&\sum\limits_{I,J}{\frac{p}{4}({\delta }_{ik}{\omega }_{lI',J}{\omega }_{jI',J}}+{\delta }_{il}{\omega }_{kI',J}{\omega }_{jI',J}+{\delta }_{jk}{\omega }_{lI',J}{\omega }_{iI',J}+{\delta }_{jl}{\omega }_{kI',J}{\omega }_{iI',J})\\
			&&+\frac{p(p-1)}{2}({{\omega }_{iI',J}}{{\omega }_{jI',J}}+{{\omega }_{iI',J}}{{\omega }_{jI',J}}) \\ 
			&&+\sum\limits_{I,J}{\frac{q}{4}({\delta }_{ik}{\omega }_{I,lJ'}{{\omega }_{I,jJ'}}+{{\delta }_{il}}{{\omega }_{I,kJ'}}{{\omega }_{I,jJ'}}+{{\delta }_{jk}}{{\omega }_{I,lJ'}}{{\omega }_{I,iJ'}}+{{\delta }_{jl}}{{\omega }_{I,kJ'}}{{\omega }_{I,iJ'}})}\\
			&&+\frac{q(q-1)}{2}({\omega }_{I,iJ'}{{\omega }_{I,jJ'}}+{\omega }_{I,iJ'}{{\omega }_{I,jJ'}}) \\ 
			&&+\frac{pq}{2}({\omega }_{li_2....i_p,iJ'}{\omega }_{kI',jJ'}+{{\omega }_{kI',iJ'}}{{\omega }_{lI',jJ'}}+{\omega }_{iI',lJ'}{\omega }_{jI',kJ'}+{\omega }_{iI',kJ'}{\omega }_{jI',lJ'}).
		\end{eqnarray*}
	\end{proposition}
	\begin{proof}
		Direct calculation yields that, for any $(p,q)$-double form
		$$\begin{aligned}
			& g((e_i\bar{\circ }e_j)\omega ,\Omega ) \\ 
			& =\frac{p}{2}\sum\limits_{I,J}{({{\omega }_{iI',J}}{{\Omega }_{jI',J}}+{{\omega }_{jI',J}}{{\Omega }_{iI',J}})}+\frac{q}{2}\sum\limits_{{}}{({{\omega }_{I,iJ'}}{{\Omega }_{I,jJ'}}+{{\omega }_{I,jJ'}}{{\Omega }_{I,iJ'}})}.
		\end{aligned}$$
		For the second part, notice that
		\begin{align*}
			g{{(({{e}_{k}}\bar{\circ }{{e}_{l}})\omega )}_{ii_2...i_p,J}}=&\frac{1}{2}({{\delta }_{ik}}{{\omega }_{li_2...i_p,J}}+{{\delta }_{il}}{{\omega }_{ki_2...i_p,J}})+\frac{1}{2}\sum\limits_{\alpha =2..p}{({{\delta }_{k{{i}_{\alpha }}}}{{\omega }_{ii_2...l...i_p,J}}+{{\delta }_{l{{i}_{\alpha }}}}{{\omega }_{ii_2...k...i_p,J}})}\\
			&+\frac{1}{2}\sum\limits_{\alpha =1..q}{({{\delta }_{k{{j}_{\alpha }}}}{{\omega }_{ii_2...i_p,j_1...l...j_q}}+{{\delta }_{l{{j}_{\alpha }}}}{{\omega }_{ii_2...i_p,j_1...k...j_q}})},
		\end{align*}
		\begin{align*}
			g{{(({{e}_{k}}\bar{\circ }{{e}_{l}})\omega )}_{I,ij_2...j_q}}=&\frac{1}{2}({{\delta }_{ik}}{{\omega }_{I,lj_2...j_q}}+{{\delta }_{il}}{{\omega }_{I,kj_2...j_q}})+\frac{1}{2}\sum\limits_{\alpha =2..q}{({{\delta }_{k{{j}_{\alpha }}}}{{\omega }_{I,ij_2...l...j_q}}+{{\delta }_{l{{j}_{\alpha }}}}{{\omega }_{I,ij_2...k...j_q}})}\\
			&+\frac{1}{2}\sum\limits_{\alpha =1..p}{({{\delta }_{k{{i}_{\alpha }}}}{{\omega }_{i_1...l...i_p,ij_2...j_q}}+{{\delta }_{l{{i}_{\alpha }}}}{{\omega }_{i_1...k...i_p,ij_2...j_q}})},
		\end{align*}
		hence
		\begin{eqnarray*}
			&&g(({{e}_{i}}\bar{\circ }{{e}_{j}})\omega ,({{e}_{k}}\bar{\circ }{{e}_{l}})\omega )) \\ 
			&=&\frac{p}{2}\left( \sum\limits_{I,J}{{{\omega }_{ji_2...i_p,J}}(\frac{1}{2}({{\delta }_{ik}}{{\omega }_{li_2...i_p,J}}+{{\delta }_{il}}{{\omega }_{ki_2...i_p,J}})+\frac{1}{2}\sum\limits_{\alpha =2..p}{({{\delta }_{k{{i}_{\alpha }}}}{{\omega }_{ii_2...l...i_p,J}}+{{\delta }_{l{{i}_{\alpha }}}}{{\omega }_{ii_2...k...i_p,J}})}}\right.\\
			&&\left.+\frac{1}{2}\sum\limits_{\alpha =1..q}{({{\delta }_{k{{j}_{\alpha }}}}{{\omega }_{ii_2...i_p,j_1...l...j_q}}+{{\delta }_{l{{j}_{\alpha }}}}{{\omega }_{ii_2...i_p,j_1...k...j_q}})}) \right) \\ 
			&&+\frac{q}{2}\left( \sum\limits_{I,J}{{{\omega }_{I,jj_2...j_q}}(\frac{1}{2}({{\delta }_{ik}}{{\omega }_{I,lj_2...j_q}}+{{\delta }_{il}}{{\omega }_{I,kj_2...j_q}})+\frac{1}{2}\sum\limits_{\alpha =2..q}{({{\delta }_{k{{j}_{\alpha }}}}{{\omega }_{I,ij_2...l...j_q}}+{{\delta }_{l{{j}_{\alpha }}}}{{\omega }_{I,ij_2...k...j_q}})}}\right.\\
			&&\left.+\frac{1}{2}\sum\limits_{\alpha =1..p}{({{\delta }_{k{{i}_{\alpha }}}}{{\omega }_{i_1...l...i_p,ij_2...j_q}}+{{\delta }_{l{{i}_{\alpha }}}}{{\omega }_{i_1...k...i_p,ij_2...j_q}}))} \right)+i,j,index,change \\ 
			&=&\sum\limits_{I,J}{\frac{p}{4}({{\delta }_{ik}}{{\omega }_{li_2...i_p,J}}{{\omega }_{ji_2...i_p,J}}+{{\delta }_{il}}{{\omega }_{ki_2...i_p,J}}{{\omega }_{ji_2...i_p,J}})}\\
			&&+\frac{p}{4}\sum\limits_{\alpha =2..p}{({{\delta }_{k{{i}_{\alpha }}}}{{\omega }_{ii_2...l...i_p,J}}{{\omega }_{ji_2...i_p,J}}+{{\delta }_{l{{i}_{\alpha }}}}{{\omega }_{ii_2...k...i_p,J}}{{\omega }_{ji_2...i_p,J}})}\\
			&&+\frac{p}{4}\sum\limits_{\alpha =1..q}{({{\delta }_{k{{j}_{\alpha }}}}{{\omega }_{ii_2...i_p,j_1...l...j_q}}{{\omega }_{ji_2...i_p,J}}+{{\delta }_{l{{j}_{\alpha }}}}{{\omega }_{ii_2...i_p,j_1...k...j_q}}{{\omega }_{ji_2...i_p,J}})} \\ 
			&&+\sum\limits_{I,J}{\frac{q}{4}({{\delta }_{ik}}{{\omega }_{I,lj_2...j_q}}{{\omega }_{I,jj_2...j_q}}+{{\delta }_{il}}{{\omega }_{I,kj_2...j_q}}{{\omega }_{I,jj_2...j_q}})}\\
			&&+\frac{q}{4}\sum\limits_{\alpha =2..q}{({{\delta }_{k{{j}_{\alpha }}}}{{\omega }_{I,ij_2...l...j_q}}{{\omega }_{I,jj_2...j_q}}+{{\delta }_{l{{j}_{\alpha }}}}{{\omega }_{I,ij_2...k...j_q}}{{\omega }_{I,jj_2...j_q}})}\\
			&&+\frac{q}{4}\sum\limits_{\alpha =1..p}{({{\delta }_{k{{i}_{\alpha }}}}{{\omega }_{i_1...l...i_p,ij_2...j_q}}{{\omega }_{I,jj_2...j_q}}+{{\delta }_{l{{i}_{\alpha }}}}{{\omega }_{i_1...k...i_p,ij_2...j_q}}{{\omega }_{I,jj_2...j_q}})}+i,j,index,change \\ 
			&=&\sum\limits_{I,J}{\frac{p}{4}({{\delta }_{ik}}{{\omega }_{li_2...i_p,J}}{{\omega }_{ji_2...i_p,J}}+{{\delta }_{il}}{{\omega }_{ki_2...i_p,J}}{{\omega }_{ji_2...i_p,J}})}\\
			&&+\frac{p(p-1)}{4}({{\omega }_{il{{i}_{3}}......i_p,J}}{{\omega }_{jk{{i}_{3}}......i_p,J}}+{{\omega }_{ik{{i}_{3}}......i_p,J}}{{\omega }_{jl{{i}_{3}}......i_p,J}})\\
			&&+\frac{pq}{4}\sum\limits_{\alpha =1..q}{({{\omega }_{ii_2...i_p,...lj_2...j_q}}{{\omega }_{ji_2...i_p,kj_2...j_q}}+{\omega }_{ii_2...i_p,kj_2......j_q}{{\omega }_{ji_2...i_p,lj_2......j_q}})} \\
			&&+\sum\limits_{I,J}{\frac{q}{4}({{\delta}_{ik}}{{\omega }_{I,lj_2...j_q}}{{\omega }_{I,jj_2...j_q}}+{{\delta }_{il}}{{\omega }_{I,kj_2...j_q}}{{\omega }_{I,jj_2...j_q}})}\\
			&&+\frac{q(q-1)}{4}({{\omega }_{I,ilj_3...j_q}}{{\omega }_{I,jkj_3...j_q}}+{{\omega }_{I,ikj_3......j_q}}{{\omega }_{I,jlj_3...j_q}})\\
			&&+\frac{pq}{4}\sum\limits_{\alpha =1..p}{({{\omega }_{li_2...i_p,ij_2...j_q}}{{\omega }_{ki_2...i_p,jj_2...j_q}}+{{\omega }_{ki_2...i_p,ij_2...j_q}}{{\omega }_{li_2...i_p,jj_2...j_q}})}+i,j,index,change\\ 
		\end{eqnarray*}
		
		\begin{eqnarray*} 
			&=&\sum\limits_{I,J}{\frac{p}{4}({{\delta }_{ik}}{{\omega }_{lI',J}}{{\omega }_{jI',J}}+{\delta }_{il}{{\omega }_{kI',J}}{{\omega }_{jI',J}}+{\delta }_{jk}{\omega }_{lI',J}{\omega }_{iI',J}+{{\delta }_{jl}}{\omega }_{kI',J}{{\omega }_{iI',J}})}\\
			&&+\frac{p(p-1)}{2}({\omega }_{ilI'',J}{\omega }_{jkI'',J}+{\omega }_{ikI'',J}{\omega }_{jlI'',J}) \\ 
			&&+\sum\limits_{I,J}{\frac{q}{4}({{\delta }_{ik}}{{\omega }_{I,lJ'}}{{\omega }_{I,jJ'}}+{{\delta }_{il}}{{\omega }_{I,kJ'}}{{\omega }_{I,jJ'}}+{{\delta }_{jk}}{{\omega }_{I,lJ'}}{{\omega }_{I,iJ'}}+{{\delta }_{jl}}{{\omega }_{I,kJ'}}{{\omega }_{I,iJ'}})}\\
			&&+\frac{q(q-1)}{2}({{\omega }_{I,ilJ''}}{{\omega }_{I,jkJ''}}+{{\omega }_{I,ikJ''}}{{\omega }_{I,jlJ''}}) \\ 
			&&+\frac{pq}{2}({{\omega }_{lI',iJ'}}{{\omega }_{kI',jJ'}}+{{\omega }_{kI',iJ'}}{{\omega }_{lI',jJ'}}+{{\omega }_{iI',lJ'}}{{\omega }_{jI',kJ'}}+{{\omega }_{iI',kJ'}}{{\omega }_{jI',lJ'}}).
		\end{eqnarray*}					
	\end{proof}
	
	\begin{proposition}\label{proposition2}
		If $\omega$ is a $(p,q)$-double form, then 
		\begin{align*}
			g(\mathring{R}({\omega}^{S_0^2}),{\omega}^{S_0^2})=&(\frac{p}{2}+\frac{2p(p+q)}{n}){{R}_{ij}}{{\omega }_{iI',J}}{{\omega }_{jI',J}}-\frac{3}{2}\frac{p(p-1)}{2}{\omega }_{ijI'',J}{\omega }_{klI'',J}R_{ijkl}\\
			&+(\frac{q}{2}+\frac{2q(p+q)}{n})R_{ij}{\omega }_{I,iJ'}{\omega }_{I,jJ'}-\frac{3}{2}\frac{q(q-1)}{2}{\omega }_{I,ijJ''}{\omega }_{I,klJ''}R_{ijkl}\\
			&+pq({\omega }_{iI',lJ'}{\omega }_{jI',kJ'}+{\omega }_{iI',kJ'}{\omega }_{jI',lJ'})R_{kijl} \\  
			& -\frac{(p+q)^2}{n^2}scal\left| \omega  \right|^2.
		\end{align*}
	\end{proposition}
	\begin{proof}
		According to \cite[Proposition 1.8]{NPW23}, we can consider $g(\bar{R}({\omega}^{S_0^2}),{\omega}^{S_0^2})$. For a $(p,q)$-double form $\omega$, since $\omega^{S_0^2}=\omega^{S^2}-\frac{p+q}{n}\omega \otimes g $ and $\bar{R}(g)=-Ric$, we have
		\begin{equation}\label{term1}
			g(\bar{R}({{\omega }^{{{S}_{0}^2}}}),{{\omega }^{{{S}_{0}^2}}})=g(\bar{R}({{\omega }^{{{S}^{2}}}}),{{\omega }^{{{S}^{2}}}})+\frac{2(p+q)}{n}g({{\omega }^{{{S}^{2}}}},\omega \otimes Ric)-\frac{(p+q)^2}{n^2}scal\left| \omega  \right|^2 .
		\end{equation}
		By Bianchi identity 
		\[{{\omega }_{ikI'',J}}{{\omega }_{jlI'',J}}R_{ijkl}={{\omega }_{ikI'',J}}{{\omega }_{jlI'',J}}R_{ilkj}+{{\omega }_{ikI'',J}}{{\omega }_{jlI'',J}}R_{ikjl}\]
		i.e.,
		\[{{\omega }_{ikI'',J}}{{\omega }_{jlI'',J}}R_{ijkl}=\frac{1}{2}{{\omega }_{ijI'',J}}{{\omega }_{klI'',J}}R_{ijkl}.\]
		This together with Proposition~\ref{compute} yields 
		\begin{align*}
			g(\bar{R}({\omega }^{S^2}),{\omega }^{S^2})=&\frac{1}{2}\sum\limits_{ijkl}{g((e_i\bar{\circ}e_j)\omega ,(e_k\bar{\circ }e_l)\omega ))(R_{kijl}+R_{kjil})} \\ 
			=&\frac{p}{2}R_{ij}{\omega }_{iI',J}{\omega }_{jI',J}-\frac{3}{2}\frac{p(p-1)}{2}{\omega }_{ijI'',J}{\omega }_{klI'',J}R_{ijkl}\\
			&+\frac{q}{2}R_{ij}{\omega }_{I,iJ'}{{\omega }_{I,jJ'}}-\frac{3}{2}\frac{q(q-1)}{2}{\omega }_{I,ijJ''}{\omega }_{I,klJ''}{{R}_{ijkl}}\\
			&+\frac{pq}{4}({{\omega }_{iI',lJ'}}{{\omega }_{jI',kJ'}}
			+{{\omega}_{iI',kJ'}}{{\omega }_{jI',lJ'}})(R_{kijl}+R_{kjil}) \\ 
			&+\frac{pq}{4}({{\omega }_{lI',iJ'}}{{\omega }_{kI',jJ'}}+{{\omega }_{kI',iJ'}}{{\omega }_{lI',jJ'}})(R_{kijl}+R_{kjil})\\
			=&\frac{p}{2}R_{ij}{\omega }_{iI',J}{\omega }_{jI',J}-\frac{3}{2}\frac{p(p-1)}{2}{\omega }_{ijI'',J}{\omega }_{klI'',J}R_{ijkl}\\
			&+\frac{q}{2}R_{ij}{\omega }_{I,iJ'}{\omega }_{I,jJ'}-\frac{3}{2}\frac{q(q-1)}{2}{\omega }_{I,ijJ''}{\omega }_{I,klJ''}R_{ijkl}\\
			&+pq({{\omega }_{iI',lJ'}}{{\omega }_{jI',kJ'}}+{{\omega}_{iI',kJ'}}{{\omega }_{jI',lJ'}})R_{kijl}.
		\end{align*}
		
		Furthermore, by using the same method as in \cite[proposition 2.3]{NPW23}, we obtain
		\begin{equation*}\label{Ricci}
			g({{\omega}^{S^2}},\omega \otimes \Ric)=p\sum\limits_{i,j}{\sum\limits_{I',J}{R_{ij}{\omega }_{iI',J}{\omega }_{jI',J}}}+q\sum\limits_{i,j}{\sum\limits_{I,J'}{R_{ij}{\omega }_{I,iJ'}{\omega }_{I,jJ'}}}.
		\end{equation*}
		Substituting the above two expressions into (\ref{term1}) completes the proof.
	\end{proof}
	\begin{proof}[\textbf{Proof of proposition~\ref{proposition1}}]
		From equation (\ref{Bochner}), we obtain
		\begin{align} 
			g(\Ric_L(\omega),\omega)&= \sum_{a=1}^p R_{i_a}{}^k \omega_{i_1 \dots \underset{(a)}{k} \dots i_p, \, J}\omega_{I,J} + \sum_{b=1}^q R_{j_b}{}^k \omega_{I, \, j_1 \dots \underset{(b)}{k} \dots j_q}\omega_{I,J} \nonumber\\ 
			&\quad - 2 \sum_{1 \le a < b \le p} R_{i_a}{}^m{}_{i_b}{}^n \omega_{i_1 \dots \underset{(a)}{m} \dots \underset{(b)}{n} \dots i_p, \, J}\omega_{I,J} \nonumber\\ 
			&\quad - 2 \sum_{1 \le a < b \le q} R_{j_a}{}^m{}_{j_b}{}^n \omega_{I, \, j_1 \dots \underset{(a)}{m} \dots \underset{(b)}{n} \dots j_q}\omega_{I,J} \nonumber\\ 
			&\quad - 2 \sum_{a=1}^p \sum_{b=1}^q R_{i_a}{}^m{}_{j_b}{}^n \omega_{i_1 \dots \underset{(a)}{m} \dots i_p, \, j_1 \dots \underset{(b)}{n} \dots j_q}\omega_{I,J}\\ 
			& =pR_{ij}{{\omega }_{iI',J}}{{\omega }_{jI',J}}+qR_{ij}{\omega }_{I,iJ'}{\omega }_{I,jJ'} \nonumber\\ 
			&\quad -p(p-1)R_{ijkl}{\omega }_{jlI'',J}{\omega }_{ikI'',J}-q(q-1)R_{ijkl}{\omega }_{I,jlJ''}{\omega }_{I,ikJ''} \nonumber\\ 
			&\quad -2pqR_{ijkl}{{\omega }_{iI',kJ'}}{{\omega }_{jI',lJ'}} \nonumber.
		\end{align}
		This together with proposition~\ref{proposition2} yields
		\begin{align*}
			\frac{3}{2}g(\Ric_L(\omega ),\omega )=&g(\bar{R}({\omega }^{S_0^2}),{\omega }^{S_0^2})\\
			&+\left( p-\frac{2(p+q)p}{n} \right)R_{ij}{{\omega }_{iI',J}}{{\omega }_{jI',J}}+\left(q-\frac{2(p+q)q}{n} \right){R_{ij}}{{\omega }_{I,iJ'}}{{\omega }_{I,jJ'}} \\ 
			& +\frac{(p+q)^2}{n^2}scal\left| \omega  \right|^2-pq({\omega }_{iI',lJ'}{\omega }_{jI',kJ'}+{\omega }_{iI',kJ'}{\omega }_{jI',lJ'})R_{kijl}\\
			&-3pqR_{ijkl}{\omega }_{iI',kJ'}{{\omega }_{jI',lJ'}}.
		\end{align*}
		By Bianchi identity
		\begin{align*}
			& {\omega }_{iI',kJ'}{\omega }_{jI',lJ'}R_{kijl} \\ 
			& ={\omega }_{iI',kJ'}{\omega }_{jI',lJ'}(R_{kjil}+R_{klji}) \\ 
			& ={\omega }_{iI',lJ'}{\omega }_{jI',kJ'}R_{kijl}-{\omega }_{iI',kJ'}{\omega }_{jI',lJ'}R_{ijkl}, 
		\end{align*}
		we obtain
		\begin{align*}
			\frac{3}{2}g(\Ric_L(\omega ),\omega )=&g(\mathring{R}({\omega}^{S_0^2}),{\omega }^{S_0^2})+\left( p-\frac{2(p+q)p}{n} \right)R_{ij}{\omega }_{iI',J}{\omega }_{jI',J}\\
			&+\left( q-\frac{2(p+q)q}{n} \right)R_{ij}{\omega }_{I,iJ'}{\omega }_{I,jJ'}+\frac{(p+q)^2}{n^2}scal\left| \omega  \right|^2\\
			&-2pqR_{kijl}{{\omega }_{iI',lJ'}}{\omega }_{jI',kJ'}-2pqR_{ijkl}{\omega }_{iI',kJ'}{\omega }_{jI',lJ'}. 
		\end{align*}
		This completes the proof the first part of Proposition~\ref{proposition1}.
		Now let us prove  (\ref{pqbochner}). Set 
		$$W^{I',J'}=({\omega }_{iI',jJ'}+{\omega }_{jI',iJ'})e_i\bar{\circ }e_j.$$
		If $\boldsymbol{c}\omega=0$,  
		then $W^{I',J'}$ is a trace-free symmetric $2$-tensor. Hence we get
		\begin{eqnarray}\label{tensor}
			&&g(\mathring{R}(W^{I',J'}),W^{I',J'})\nonumber\\
			&=&({\omega }_{iI',jJ'}+{\omega }_{jI',iJ'})({\omega }_{kI',lJ'}+{\omega }_{lI',kJ'})g(\bar{R}(e_i\bar{\circ} e_j),e_k\bar{\circ} e_l) \nonumber\\ 
			&=&\frac{1}{2}\left({\omega }_{iI',jJ'}{\omega }_{kI',lJ'}+{\omega }_{iI',jJ'}{\omega }_{lI',kJ'}+{\omega }_{jI',iJ'}{\omega }_{kI',lJ'}+{\omega }_{jI',iJ'}{\omega }_{lI',kJ'} \right)(R_{iklj}+R_{ilkj}) \nonumber\\ 
			&=&-2{\omega }_{iI',jJ'}{\omega }_{kI',lJ'}R_{ikjl}-2{\omega}_{iI',jJ'}{\omega }_{lI',kJ'}R_{ikjl}.
		\end{eqnarray}
	\end{proof}

	\section{proof of theorem B, C and  D}

	To apply the weighting principle, we need to compute the highest weight and the totally weight corresponding to each term in equation (\ref{pqbochner}).
	\begin{proposition}\label{weighting estimate}
		For a $(p,q)$-double form $\omega$ satisfying Bianchi identity and $S\in S_0^2(V)$,
		\begin{align*}
			&(i)\left| {{\omega }^{S_{0}^{2}}} \right|^2=[\frac{n+2}{2}(p+q)-\frac{n+2}{2n}{{(p+q)}^{2}}+(p+1)q]\left| \omega  \right|^2+pq\left| \boldsymbol{c}\omega  \right|^2-\frac{q}{p+1}|\mathcal{B}\omega|^2,\\
			&(ii){{\left| S\omega  \right|}^{2}}\le \left( p+q+2\min \{p,q\}-\frac{{{\left( p+q \right)}^{2}}}{n} \right){{\left| S \right|}^2}\left| \omega  \right|^2.
		\end{align*}
	\end{proposition}
	\begin{proof}
		(i) The norm $\left| {\omega }^{S_0^2}\right|^2$ can be computed by taking $R_{ijkl}=\delta_{ik}\delta_{jl}-\delta_{il}\delta_{jk}$ in Proposition~\ref{proposition2}. We get
		$$\left| {{\omega }^{S_{0}^{2}}} \right|^2=[\frac{n+2}{2}(p+q)-\frac{n+2}{2n}{{(p+q)}^{2}}+pq]\left| \omega  \right|^2+pq\left| \boldsymbol{c}\omega  \right|^2+pq{\omega}_{iI',jJ'}{\omega }_{jI',iJ'}$$
		and notice that from (\ref{Bianchi})
		\begin{align*}
			\left| \mathcal{B}\omega  \right|^2 =(p+1)\left| \omega  \right|^2-p(p+1){\omega }_{iI',jJ'}{\omega }_{jI',ij_2....j_q}.  
		\end{align*}
		
		(ii) For $S\in S_{0}^{2}(V)$, there is an orthonormal basis $e_1,\ldots ,e_n$ for V and $a_1,\ldots ,a_n\in \mathbb{R}$ such that $Se_i=a_ie_i$ for $i=1,\ldots n$. It follows that
		\[{{\left( S\omega  \right)}_{I,J}}=\left( \sum\limits_{i}{c_ia_i} \right){\omega}_{I,J},\]
		where $c_i=\left\{ \begin{aligned}
			& 0\qquad \text{for} \quad i\in {\{1,\ldots n\}}/{I\cup J}\; \\ 
			& 1\qquad \text{for} \quad i\in {I\cup J}/{I\cap J}\; \\ 
			& 2\qquad \text{for} \quad i\in I\cap J 
		\end{aligned} \right.$.
		
		Now we estimate maximum of $\left( \sum\limits_{i}{c_ia_i} \right)^2$ under constraints
		$\sum\limits_{i}{a_i}=0$.
		If we set $\left| I\cap J \right|=r$, then $\sum\limits_{i}{c_i}=p+q$ and $\sum\limits_{i}{c_i^2}=p+q+2r$. Choose $\mu=\frac{p+q}{n}$, then by  the Cauchy-Schwarz inequality
		
		\begin{align*}
			\left( \sum\limits_{i}{c_ia_i} \right)^2=&\left( \sum\limits_{i}{(c_i-\mu )a_i }\right)^2\le \left( \sum\limits_{i}{(c_i-\mu )^2} \right)\left( \sum\limits_{i}{a_i^2} \right) \\ 
			& =\left(\sum\limits_{i}{(c_i-\mu )^2} \right)|S|^2\\ 
			& =\left( p+q+2r-\frac{(p+q)^2}{n} \right)|S|^2\\ 
			& \le \left(p+q+2\min \{p,q\}-\frac{(p+q)^2}{n}\right)|S|^2.
		\end{align*}
		Therefore,
		\[|S\omega|^2=\sum_{I,J}\left( \sum\limits_{i}{c_ia_i} \right)^2{\omega}_{I,J}^2\le \left( p+q+2\min \{p,q\}-\frac{{{\left( p+q \right)}^{2}}}{n} \right){{\left| S \right|}^2}\left| \omega  \right|^2.\]
	\end{proof}
	\begin{remark}
		The estimate in Proposition~\ref{weighting estimate} is sharp if we choose, for example, $\omega = e^1 \wedge\cdots\wedge e^p \otimes e^1 \wedge\cdots\wedge e^p $ and $S\in S_0^2(V)$ given by $S(e_i) = a_ie_i$ with 
		\begin{align*}
			a_1=\cdots=a_p=|S|\sqrt{\frac{n-p}{np}},\\
			a_{p+1}=\cdots=a_n=|S|\sqrt{\frac{p}{n(n-p)}}.
		\end{align*}
	\end{remark}

	By proposition~\ref{weighting estimate} we have
	\begin{equation}\label{weight1}
		g(\mathring{R}({\omega}^{S_0^2}),{\omega }^{S_0^2})\ge \left[\mathring{R},p+q+2\min \{p,q\}-\frac{(p+q)^2}{n},\frac{n+2}{2}(p+q)-\frac{n+2}{2n}(p+q)^2+pq+q \right]\left|\omega\right|^2.
	\end{equation}
	
	Since $tr(\mathring{R})=\frac{n+2}{2n}scal$, the scalar curvature of $M$ satisfies 
	\begin{equation}\label{scalar}
		scal\ge\frac{2n}{n+2}[\mathring{R},1,\frac{(n-1)(n+2)}{2}].
	\end{equation}
	It follows from \cite[Lemma 3.14]{NPW23} and \cite{FL1}, the Ricci curvature of $M$ satisfies 
	\begin{equation}\label{Ricci1}
		Ric\ge \left[ \mathring{R},\frac{n}{n+2},n-1 \right],
	\end{equation}
	and for $p \ge 2$, 	
	\begin{align}\label{Ricci2}
		\sum\limits_{i=1}^{p}{R_{ii}}&\ge \frac{n-p+1}{n-p+2}[\mathring{R},2,p(n-1)]+\frac{p}{n(n-p+2)}scal \nonumber\\ 
		& \ge \frac{n-p+1}{n-p+2}[\mathring{R},2,p(n-1)]+\frac{p}{n(n-p+2)}\frac{2n}{n+2}[\mathring{R},1,\frac{(n-1)(n+2)}{2}] \nonumber\\ 
		& \ge [\mathring{R},\frac{2(n+1)}{n+2},p(n-1)]. 
	\end{align}
	
	Let $\{S_\alpha\}_{\alpha=1}^N$ be an orthonormal eigenbasis of $\mathring R$ with corresponding eigenvalues $\lambda_\alpha$. then the last term of (\ref{pqbochner}) can be written as
	\[
	\sum_{I',J'}{g(\mathring{R}(W^{I',J'}),W^{I',J'})}=\sum_{\alpha ,I',J'}{\lambda_\alpha \left(W_{\alpha }^{I'.J'} \right)^2}.
	\]
	Furthermore, applying (\ref{tensor}) to the curvature tensor $R_{ijkl}={\delta }_{ik}{\delta }_{jl}-{\delta }_{il}{\delta }_{jk}$ gives
	\begin{equation*}
		\sum\limits_{\alpha ,I',J'}\left( W_{\alpha }^{I',J'} \right)^2=\left(2+\frac{2}{p} \right)\left|\omega \right|^2.
	\end{equation*}
	According to an elementary Lagrange multiplier argument we have
	\begin{equation}\label{tensor1}
		\sum_{I',J'}{g(\mathring{R}(W^{I',J'}),W^{I',J'})}\ge \left(2+\frac{2}{p}\right)\left[\mathring{R},\frac{n^2+n-4}{n^2+n-2},1\right]\left|\omega\right|^2.
	\end{equation}
	Combining (\ref{weight1})-(\ref{tensor1}), we obtain the following estimate. 	
	\begin{proposition}\label{estimate}
		Let $\omega$ be a $(p,q)$ double form with $p,q\ge 0$ and $p+q\le \frac{n}{2}$,  satisfying the first Bianchi identities and its first contraction vanishing. Then
		\begin{align*}
			\frac{3}{2}g(\Ric_L(\omega),\omega)\ge\left[\mathring{R}, A(p,q,n),\frac{3}{2}[n(p+q)-p^2-q^2]+3q\right]|\omega|^2.
		\end{align*}
	\end{proposition}
	\begin{proof}
		The case $p,q\le 2$ follows from substituting (\ref{weight1}),(\ref{scalar}),(\ref{Ricci1}),(\ref{tensor1}) into (\ref{pqbochner}) and applying Lemma~\ref{weight principle}.
		
		The case $p,q\ge 3$ follows similarly from substituting (\ref{weight1}),(\ref{scalar}),(\ref{Ricci2}),(\ref{tensor1}) into (\ref{pqbochner}) and applying Lemma~\ref{weight principle}
		
		The case $p\ge 3,q\le 2$ (similarly, $q\ge 3,p\le 2$),  for the term $\left( p-\frac{2(p+q)p}{n} \right)R_{ij}{\omega }_{iI',J}{\omega }_{jI',J}$ in (\ref{pqbochner}) we use (\ref{Ricci2}) and (\ref{Ricci1}) for the term $\left( q-\frac{2(p+q)q}{n} \right)R_{ij}{\omega }_{I,iJ'}{\omega }_{I,jJ'}$.
		
	\end{proof}
	\begin{remark}
		This is a generalization of \cite[Proposition 3.15]{NPW23} where we identity $p$-form as $(p,0)$ double form and \cite[Proposition 3.5]{FL1} for $p=q=2$.
	\end{remark}

	\begin{proof}[Proof of Theorem B]							
		The fact that all harmonic $(p,q)$ double forms are vanishing (resp. $\omega $ must be parallel) if the curvature operator of the second kind is $C(p,q,n)$-positive (resp. $C(p,q,n)$-nonnegative) is a direct consequence of the weight principle, Proposition~\ref{estimate} and the Bochner technique	
	\end{proof}
	
	\begin{proof}[Proof of Theorem C]
		According to \cite{Kulkarni,Labbi} any $(p, p)$ double form $\omega$ has an orthogonal decomposition $\omega = g^i\omega_{p-i}$
		where $g^i$ are the exterior powers of the metric $(1, 1)$ double form $g$ and the $\omega^k$ are $(k, k)$ double
		forms such that $c\omega^k = 0$. It follows from \cite{Kulkarni} that if $\omega$ is symmetric and satisfies the first Bianchi identity
		then so are all the components $\omega_i$ for $1 \le i \le p$.  By applying Theorem~\ref{theorem-A} to the distinct $\omega^k$ components, we deduce that $\omega^k = 0$ for every $k$ with $1 \le k \le j$ if $\mathring{R}$ is $C'(j,n)$-positive. Conversely, the vanishing of $\omega^k$ for all $1 \le k \le j$ is equivalent to the condition $c^{p-j}\omega = \lambda g^j$ with a constant $\lambda$, as established in \cite[Proposition 2.2]{Labbi2}.
	\end{proof}
	
	\begin{proof}[Proof of Theorem D]
		The orthogonal decomposition of the double form $\omega $ takes the form
		\[\omega ={\omega }_p+{g^p}{{\omega }_{0}}={{\omega }_{p}}+\frac{\lambda }{p(n-p+1)}{g^p}.\]
		When $\omega$ is harmonic, it follows from \cite{Labbi3} that ${\omega}_p$ is harmonic. Under the $C(p,n)$-positive condition, Theorem A shows ${\omega }_{p}$ vanish and hence $\omega$ has constant sectional curvature. If $C(p,n)$-nonnegative, then $\omega $ is parallel.
	\end{proof}

\end{document}